\documentclass[a4paper]{amsart}
\usepackage{amsfonts, amssymb, amsmath, eucal, verbatim, amsthm, mathrsfs, amscd, enumerate, ascmac, fancyhdr, caption, hyperref}
\usepackage[whole]{bxcjkjatype}

\usepackage[dvipdfmx,svgnames]{xcolor}
\usepackage{tikz}
\usetikzlibrary{positioning}
\usetikzlibrary{intersections, calc}
\usepackage{amsmath}
\usetikzlibrary{arrows}

\usepackage{pgfplots}

\newtheorem{theorem}{Theorem}
\newtheorem{lemma}[theorem]{Lemma}
\newtheorem{proposition}[theorem]{Proposition}
\newtheorem*{conjecture*}{Conjecture}

\newtheorem{corollary}[theorem]{Corollary}
\theoremstyle{definition}

\newtheorem*{goal*}{Goal}

\theoremstyle{remark}

\newtheorem*{remarks}{Remarks}

\def\supp{\mathop{\mathrm{supp}}\nolimits}

\author{Shobu Shiraki}
\address[Shobu Shiraki] {Department of Mathematics, Graduate school of Science and Engineering, Saitama University, Saitama, 338-8570, Japan}
\email{s.shiraki.446@ms.saitama-u.ac.jp}

\begin{document}
\date{\today}
\title[Pointwise convergence along restricted directions]{Pointwise convergence along restricted directions for the fractional Schr\"odinger equation}

\begin{abstract}
We consider the pointwise convergence problem for the solution of Schr\"odinger-type equations along directions determined by a given compact subset of the real line. This problem contains Carleson's problem as the most simple case and was studied in general by Cho--Lee--Vargas. We extend their result from the case of the classical Schr\"odinger equation to a class of equations which includes the fractional Schr\"odinger equations. To achieve this, we significantly simplify their proof by completely avoiding a time localization argument.
\end{abstract}
\maketitle
\section{Introduction}
Let $d\ge1$, $a>1$ and consider the fractional Schr\"odinger equation
\[
\begin{cases}
\partial_tu(x,t)=i(-\Delta_x)^{\frac a2}u(x,t)&\quad(x,t)\in\mathbb{R}^d\times\mathbb{R}\\
u(x,0)=f(x)&\quad x\in\mathbb{R}^d.
\end{cases}
\]
It is well-known that for a sufficiently nice initial data $f$, the solution can be written as
\[
u(x,t)=e^{it(-\Delta)^{\frac a2}}f(x):=\left(\frac{1}{2\pi}\right)^d\int_{\mathbb{R}^d}e^{i(x\cdot\xi+t|\xi|^a)}\widehat{f}(\xi) \,d\xi,
\]
where $\widehat{f}(\xi):=\int_{\mathbb{R}^d}e^{-ix\cdot\xi}f(x)\,dx$. When $a=2$, this is the standard Schr\"odinger equation from quantum mechanics. The general case has arisen in recent years in physical models and has been the subject of numerous papers (see, for example, \cite{CL13, CS88, GH11, GY14, Ls00, Ps09}). Associated with the fractional Schr\"odinger equation, it is natural to try to determine the minimum level of regularity $s$ which guarantees that the limit
\begin{equation}\label{lim:general Carleson problem with Gamma}
\lim_{\substack{(y,t)\to (x,0)\\(y,t)\in\Gamma_x}}e^{it(-\Delta)^{\frac a2}}f(y)=f(x)\qquad \mathrm{a.e.}
\end{equation}
holds whenever $f\in H^s(\mathbb{R^d})$. Here, $H^s(\mathbb{R}^d)$ is the Sobolev space of order $s$ whose norm is given by 
\[
\|f\|_{H^s(\mathbb{R}^d)}=\|(1-\Delta)^{\frac s2}f\|_{L^2(\mathbb{R}^d)}
\]
and $\Gamma_x\subset\mathbb{R}^d\times[-1,1]$ is a convergence domain correspondings to each $x\in\mathbb{R}^d$.

The classical case, known widely as Carleson's problem, is concerned with the case of vertical lines $\Gamma_x=\{x\}\times[-1,1]$. Here, when $d=1$ it is known that \eqref{lim:general Carleson problem with Gamma} holds if and only if $s\ge\frac14$; see the work of Carleson \cite{Cr80} and Dahlberg--Kenig \cite{DK82} for the case $a=2$, and also see the work of Sj\"olin \cite{Sj87} for $a>1$. The higher dimensional case $d\ge2$ has been subject to a recent flurry of activity. When $a=2$, Bourgain \cite{Br16} showed that $s\ge\frac12-\frac{1}{2(d+1)}$ is necessary for \eqref{lim:general Carleson problem with Gamma} for $d\ge2$, and Du--Guth--Li \cite{DGL17} and Du--Zhang \cite{DZ18} have shown $s>\frac12-\frac{1}{2(d+1)}$ is sufficient for \eqref{lim:general Carleson problem with Gamma} for $d=2$ and $d\ge3$, respectively (for important earlier contributions see, for example, papers by Vega \cite{Vg88}, Lee \cite{Lee06} and Bourgain \cite{Br13}). In addition, for $a>1$, Cho--Ko \cite{CK18} proved that \eqref{lim:general Carleson problem with Gamma} holds if $s>\frac12-\frac{1}{2(d+1)}$ and $d\ge2$. 

Non-tangential convergence corresponds to the case
\[
\Gamma_x=\{x+t\theta : t\in[-1,1]\ \text{and}\ \theta\in \mathbb{B}\},
\]
where $\mathbb{B}\subset\mathbb{R}^d$ is a given euclidean ball which is centered at the origin, that is, $\Gamma_x$ is a conical region with vertex at $(x,0)$ and aperture determined by the radius of $\mathbb{B}$. In this case, it is known that \eqref{lim:general Carleson problem with Gamma} holds if and only if $s>\frac d2$. The sufficiency part of this claim follows easily by a well-known argument using Sobolev embedding and the delicate necessity part has proved by Sj\"ogren--Sj\"olin in \cite{SS89}.

When $d=1$, the classical case and the non-tangential case were unified in a natural way by Cho--Lee--Vargas \cite{CLV12} who proved that \eqref{lim:general Carleson problem with Gamma} holds in the case 
\[
\Gamma_x=\{x+t\theta: t\in[-1,1]\ \text{and}\ \theta \in \Theta\}
\]
when $a=2$ and $s>\frac{\beta(\Theta)+1}{4}$. Here, $\Theta\subset\mathbb{R}$ is a given compact set and $\beta(\Theta)$ denotes the upper Minkowski dimension of $\Theta$. Our main goal in this paper is to improve the result in \cite{CLV12} by extending to a class of equations which includes the fractional Schr\"odinger equation for $a>1$. We define the evolution operator $S_t$ on appropriate input functions by
\[
S_tf(x)=\frac{1}{2\pi}\int_{\mathbb{R}}e^{i(x\xi+t\Phi(\xi))}\widehat{f}(\xi)\,d\xi.
\]
Here, $\Phi:\mathbb{R}\to\mathbb{R}$ is a $C^2$ function which satisfies
for some $C_1>0$,
\begin{equation}\label{Phi'' is bounded below}
|\xi||\Phi''(\xi)|\ge C_1 
\end{equation}
for all $|\xi|\ge1$. Moreover, for some $C_2>0$,
\begin{equation}\label{Phi'' is bigger than Phi'}
|\xi||\Phi''(\xi)|\ge C_2|\Phi'(\xi)|
\end{equation}
 for all $|\xi|\ge1$. It is trivial to verify that $\Phi(\xi)=|\xi|^a$ satisfies these conditions when $a>1$.
 
Our main result is the following.
\begin{theorem}\label{thm:improved CLV}
Let $\Theta\subset\mathbb{R}$ be compact and suppose $\Phi\in C^2(\mathbb{R})$ satisfies \eqref{Phi'' is bounded below} and \eqref{Phi'' is bigger than Phi'}. For any $q\in[1,4]$ and $s>\frac{\beta(\Theta)+1}{4}$, there exists a constant $C_{q,s}$ such that 
\[
\left\|\sup_{(t,\theta)\in[-1,1]\times\Theta}|S_tf(\cdot+t\theta)|\right\|_{L^q(-1,1)}\le C_{q,s}\|f\|_{H^s(\mathbb{R})}.
\]
\end{theorem}
By standard arguments, we thus obtain the associated pointwise convergence.
\begin{corollary}
Let $\Theta\subset\mathbb{R}$ be compact and suppose $\Phi\in C^2(\mathbb{R})$ satisfies \eqref{Phi'' is bounded below} and \eqref{Phi'' is bigger than Phi'}. If $s>\frac{\beta(\Theta)+1}{4}$, then
 \begin{equation}
\lim_{\substack{(y,t)\to (x,0)\\y-x\in t\Theta}}S_tf(y)=f(x)\qquad \mathrm{a.e.}
\end{equation}
whenever $f\in H^s(\mathbb{R})$. 
\end{corollary}
Theorem \ref{thm:improved CLV} improves the result in \cite{CLV12} in two respects; the class of evolution operators has been widened from the case $\Phi(\xi)=|\xi|^2$ to those satisfying \eqref{Phi'' is bounded below} and \eqref{Phi'' is bigger than Phi'}, and our maximal estimates are valid for $q\in[1,4]$ (the estimate in \cite{CLV12} was proved in the only case $q=2$). While the proof in \cite{CLV12} may be modified in a straightforward way to go beyond the classical case $\Phi(\xi)=|\xi|^2$ to a certain extent, it seems to us to be difficult to handle case $\Phi(\xi)=|\xi|^a$ with $a$ close to $1$. Indeed, the argument in \cite{CLV12} rests on a certain widely used time localization argument which becomes increasingly weak as $a$ approaches $1$. To overcome this significant obstacle, we remove the use of the time localization lemma; this simplification to the proof has allowed us to handle the case $\Phi(\xi)=|\xi|^a$ for any $a>1$. Further explanation of this point will follow our proof of Theorem \ref{thm:improved CLV} in Section \ref{sec:Proof of Theorem}. Prior to that, we prepare for the proof of Theorem \ref{thm:improved CLV} in Section \ref{sec:preliminaries}.

\section{Preliminaries}\label{sec:preliminaries}
\subsection*{Notations}  
Associated with the operator $S_t$ given above by 
\[
S_tf(x)=\frac{1}{2\pi}\int_{\mathbb{R}}e^{i(x\xi+t\Phi(\xi))}\widehat{f}(\xi)\,d\xi
\]
and a fixed compact set $\Theta \subset \mathbb{R}$, we define the maximal operator $M_\Theta$ by
\[
M_\Theta f(x)=\sup\{|S_tf(x+t\theta)| : -1\le t\le 1,\,\theta\in\Theta\}.
\]
Also, we recall that the upper Minkowski dimension of $\Theta$ is defined by
\[
\beta(\Theta)=\inf\{r>0:\limsup_{\delta\to0}N(\Theta,\delta)\delta^r=0 \},
\]
where $N(\Theta,\delta)$ denotes the smallest number of $\delta$-intervals which cover $\Theta$.

We will use the following notations frequently:
\begin{itemize}
\item $I=(-1,1)$.

\item $q'=\frac{q}{q-1}$: H\"older conjugate of $q\in[1,\infty]$.

\item $A\lesssim B$: $A\le CB$ for some constant $C>0$.

\item $A\gtrsim B$: $A\ge CB$ for some constant $C>0$.

\item $A\sim B$: $C^{-1}B\le A\le CB$ for some constant $C>0$.

\item $L^p_xL^q_tL^r_\theta$: The Lebesgue space with norm 
\[
\|F\|_{L_x^pL_t^qL_\theta^r}=\left(\int\left(\int\left(\int|F(x,t,\theta)|^r\,d\theta\right)^{\frac qr}\,dt\right)^{\frac pq}\,dx\right)^{\frac1p},
\]
where the domains of integration will be clear from the context.

\end{itemize}

\subsection*{Useful lemmas}
The following lemmas will be crucial for the oscillatory integral estimates in the proof of Theorem \ref{thm:improved CLV}. Applying these lemmas appropriately essentially allows us to avoid the time localization lemma, which is used in \cite{CLV12}.  
\begin{lemma}[van der Corput's lemma]\label{lem:van der Corput}
Suppose $\lambda>1$ and we have $|\phi^{(k)}(x)|\geq1$ for all $(a,b)$. If $k=1$ and $\phi'$ is monotonic on $(a,b)$, or simply $k\geq2$, then there exists a constant $C_k$ such that
\[
\left|\int_a^be^{i\lambda\phi(x)}\psi(x)\,dx\right|<C_k\lambda^{-\frac1k}\left(\int_a^b|\psi'(x)|\,dx+\|\psi\|_{L^{\infty}}\right).
\]
\end{lemma}
For a proof of van der Corput's lemma, we refer the reader to \cite{St94}. 

\begin{lemma}\label{lem: HLS-type}
Let $1\leq q\le4$. There exists a constant $C_q\ge0$ such that 
\[
\left|\iiiint g(x,t)h(x',t') |x-x'|^{-\frac12}\,dxdx'dtdt'\right|\le C_q\|g\|_{L_x^{q'}L_t^1}\|h\|_{L^{q'}_xL^1_t},
\]
where the integrals are taken over $(x,t), (x',t')\in I\times[-1,1]$. 
\end{lemma}

\begin{proof}
Denoting $G(x)=\|g(x,\cdot)\|_{L^1}$ and $H(x')=\|h(x',\cdot)\|_{L^1}$,
\[
\left|\iiiint g(x,t)h(x',t') |x-x'|^{-\frac12}\,dxdx'dtdt'\right|\le \int_{-1}^1\int_{-1}^1G(x)H(y)|x-x'|^{-\frac12}\,dxdx'.
\]
By the Hardy--Littlewood--Sobolev inequality,
\begin{align*}
\int_{-1}^1\int_{-1}^1G(x)H(x')|x-x'|^{-\frac12}\,dxdx'&\lesssim\|G\|_{L^{\frac43}(I)}\|H\|_{L^{\frac43}(I)}\\
&\lesssim\|g\|_{L_x^{q'}L_t^1}\|h\|_{L_x^{q'}L_t^1},
\end{align*}
where the last inequality is obtained by H\"older's inequality since $\frac 43\le q'$ from our assumption.
\end{proof}

\section{Proof of Theorem \ref{thm:improved CLV}}\label{sec:Proof of Theorem}
\begin{proof}[Proof of Theorem \ref{thm:improved CLV}]
We fix $q\in[2,4]$ and without loss of generality, we suppose $\Theta\subset [-1,1]$. The case $q \in [1,2)$ follows immediately by H\"older's inequality.

The first half of the proof is based on the proof in \cite{CLV12}. Suppose $\psi_0\in C_0^\infty(I)$ and $\psi\in C_0^\infty((-2,-\frac12)\cup(\frac12,2))$ give rise to a standard dyadic partition of unity
$$\psi_0(\xi)+\sum_{k\ge1}\psi_{k}\equiv1,$$
where $\psi_{k}=\psi(\frac{\cdot}{2^{k-1}})$. For each $0\le k\in\mathbb{Z}$, the frequency localization operator $P_{k}$ is defined  by
\[
\widehat{P_{k}f}(\xi)=\psi_k(\xi)\widehat{f}(\xi).
\]
Then, 
\begin{equation}\label{ineq: Shiraki LP decom.}
\left\|M_{\Theta}  f\right\|_{L^q(I)}^q\lesssim\|M_{\Theta}  P_0f\|_{L^q(I)}^q+\sum_{k\ge1}\left\|M_{\Theta}  P_{k}f\right\|_{L^q(I)}^q.
\end{equation}
The first term is relatively easy to estimate. In fact,
\begin{align*}
\|M_{\Theta}P_0f\|_{L^q(I)}^q&\lesssim\int_{\mathbb{R}}\psi_0(\xi)|\widehat{f}(\xi)|\,d\xi\\
&\lesssim\|f\|_{L^2}\\
&\lesssim\|f\|_{H^s}
\end{align*}
for $s\ge0$ and thereby this case is no problem.

For the remaining cases, fix a parameter $\sigma$ (to be chosen at the end of the proof) satisfying
\begin{equation}\label{Shiraki condi:Thera to Omega}
\frac q4\le \sigma\le1.
\end{equation}
Since $\Theta$ is compact, for each $\lambda>0$ there exists a finite collection $\{\Omega_j(\lambda)\}_{j=1}^N$ which satisfies
$$
\Theta\subset\bigcup_{j=1}^N\Omega_j(\lambda),
$$
where $|\Omega_j(\lambda)|\le\lambda^{-\sigma}$ for each $j$ and $N=N(\Theta,\lambda^{-\sigma})$ is the smallest number of $\lambda^{-\sigma}$-intervals which cover $\Theta$. For fixed $k$ and $x\in(-1,1)$, 
\[
M_{\Theta}P_{k}f(x)^q
\le \sum_{j=1}^{N}\sup_{\substack{-1\le t\le1\\ \theta\in\Omega_{k,j}}}|S_t P_{k}f(x+t\theta)|^q,
\]
where $\Omega_{k,j} = \Omega_j(2^k)$, and therefore
\begin{align*}
\sum_{k\ge1}\left\|M_{\Theta}  P_{k}f\right\|_{L^q(I)}^q\le \sum_{k\ge1}\sum_{j=1}^N\left\|M_{\Omega_{k,j}}  P_{k}f\right\|_{L^q(I)}^q.
\end{align*}
Now, we shall introduce the following useful proposition.
\begin{proposition}\label{prop: MPf is lesssim than lambda f}
Let $k\ge1$ and $\Omega$ be an interval with $|\Omega|\le 2^{-\sigma k}$. Then, there exists a constant $C_q>0$ such that 
\begin{equation}\label{ineq: Shiraki MPf is less than lambda f}
\left\|M_{\Omega}P_k f\right\|_{L^q(I)}\le C_q  2^{\frac k4}\|f\|_{L^2}
\end{equation}
holds for all $f\in L^2(\mathbb{R})$.
\end{proposition}

\begin{proof}[Proof of Proposition \ref{prop: MPf is lesssim than lambda f}]
Set $\lambda=2^k$ and  
\[
Tf(x,t,\theta):=\chi(x,t,\theta)\int_{\mathbb{R}} e^{i((x+t\theta)\xi+t\Phi(\xi))}f(\xi)\psi(\tfrac\xi\lambda)\,d\xi,
\]
where $\chi=\chi_{I\times[-1,1]\times\Omega}$. Then \eqref{ineq: Shiraki MPf is less than lambda f} follows from
\begin{equation}\label{ineq: Shiraki Tf is less than lambda f}
\|Tf\|_{L^{q}_x L^\infty_t L^\infty_\theta}\lesssim\lambda^{\frac14}\|f\|_{L^2}\quad(\lambda\gtrsim1)
\end{equation}
since
\begin{align*}
\left\|M_{\Omega}  P_{k} f\right\|_{L^q(I)}&\sim\|T\widehat{f}\|_{L_x^2L_t^\infty L_\theta^\infty}\\
&\lesssim\lambda^{\frac14}\|\widehat{f}\|_{L^2}\\
&\lesssim\lambda^{\frac14}\|f\|_{L^2}
\end{align*}
by Plancherel's theorem. Let us consider the dual form of \eqref{ineq: Shiraki Tf is less than lambda f}, which is
\begin{equation}\label{ineq: Shiraki T* argument}
\|T^*F\|_{L^2}\lesssim\lambda^{\frac14}\|F\|_{L^{q'}_xL^1_tL^1_\theta}
\end{equation}
where
\[
T^*F(\xi)=\psi(\tfrac\xi\lambda)\iiint\chi(x',t',\theta') e^{-i((x'+t'\theta')\xi+t'\Phi(\xi))}F(x',t',\theta')\,dx'dt'd\theta'.
\]
Then,
\begin{align*}
&\|T^*F\|_{L^2}^2\\
&=\lambda\int\psi^2(\xi)\iiint\iiint\chi(x,t,\theta)\chi(x',t',\theta')\\
&\quad\times e^{i(\lambda(x-x'+t\theta+t'\theta')\xi+(t-t')\Phi(\lambda\xi))}\bar{F}(x,t,\theta)F(x',t',\theta')\,dxdtd\theta dx'dt'd\theta'd\xi\\
&=\lambda\int_{W}\int_{W'}\chi(w)\chi(w')\bar{F}(w)F(w')K_\lambda(w,w')\,dwdw'\\
&=\sum_{\ell=1}^3\lambda\iint_{V_\ell}\chi(w)\chi(w')\bar{F}(w)F(w')K_\lambda(w,w')\,dwdw'\\
&=:A_1+A_2+A_3.
\end{align*}
Here, we denote $w=(x,t,\theta)\in W$ and $w'=(x',t',\theta')\in W$, where $W:=I\times[-1,1]\times\Omega$. Also, 
\[
K_\lambda(w,w')=\int_{\mathbb{R}} e^{i\phi(\lambda\xi)}\psi^2(\xi)\,d\xi,
\]
\begin{align*}
\phi(\xi,w,w')=(x-x'+t\theta-t'\theta')\xi+(t-t')\Phi(\xi),
\end{align*}
and
\[
\begin{cases}
V_1&=\{(w,w')\in W\times W:|x-x'|<4|t-t'|\},\\
V_2&=\{(w,w')\in W\times W:|x-x'|\ge4|t-t'|\ \mbox{and}\ |x-x'|\ge4\lambda^{-\sigma}\},\\
V_3&=\{(w,w')\in W\times W:|x-x'|\ge4|t-t'|\ \mbox{and}\ |x-x'|<4\lambda^{-\sigma}\}.
\end{cases}
\]

Thus, \eqref{ineq: Shiraki T* argument} follows from
\[
A_\ell\lesssim\lambda^{\frac12}\|F\|^2_{L^{q'}_xL^1_tL^1_\theta} 
\]
for each $\ell=1,2,3$. 
\subsection*{The term $A_1$}
Let us start with an estimate of $A_1$. Since
\[
|\phi''(\lambda\xi)|=\lambda^2|t-t'||\Phi''(\lambda\xi)|\gtrsim\lambda|x-x'|
\]
holds from \eqref{Phi'' is bounded below}, we are allowed to apply Lemma \ref{lem:van der Corput} to get
\[
|K_\lambda(x,x',t,t',\theta,\theta')|\lesssim(\lambda|x-x'|)^{-\frac12}.
\]
By using Lemma \ref{lem: HLS-type}, it follows that 
\begin{align*}
A_1&\le\lambda^{\frac12}\iint_{V_1}\chi(w')|F(w')|\chi(w)|\bar{F}(w)||x-x'|^{-\frac12}\,dwdw'\\
&\lesssim\lambda^{\frac12}\|F\|_{L^{q'}_xL^1_tL^1_\theta}^2.
\end{align*}

\subsection*{The term $A_2$}
Next, we shall consider $A_2$. In this case, the following key relationship holds:
\begin{equation} \label{relationship: phase for A2}
|x-x'+t\theta-t'\theta'|\sim|x-x'|.
\end{equation}
Indeed, 
\begin{align*}
|x-x'+t\theta-t'\theta'|&\ge |x-x'|-|t-t'|-|\theta-\theta'|\\
&\ge\frac34|x-x'|-\lambda^{-\sigma}\\
&\ge\frac12|x-x'|.
\end{align*}
Similarly, the other way holds, too. 

Now, let us observe that for all $(w,w')\in V_2$,  we have
\begin{equation}\label{ineq:K for V2}
|K_\lambda(w,w')|\lesssim\lambda|x-x'|^{-\frac12}.
\end{equation}
Before proving \eqref{ineq:K for V2}, we note that 
\begin{align*}
A_2&\lesssim\lambda^{\frac12}\|F\|_{L^{q'}_xL^1_tL^1_\theta}^2
\end{align*}
immediately follows by using Lemma \ref{lem: HLS-type} as before.

To see \eqref{ineq:K for V2}, let us split $K_\lambda$ into $B_1$ and $B_2$ as follows
\begin{align*}
K_\lambda(w,w')&=\int_{U_1}e^{i\phi(\lambda\xi)}\psi^2(\xi)\,d\xi +\int_{U_2}e^{i\phi(\lambda\xi)}\psi^2(\xi)\,d\xi\\
&=:B_1+B_2,
\end{align*}
where
\[
U_1=\{\xi\in\mathbb{R}:|x-x'+t\theta-t'\theta'|\ge2|t-t'||\Phi'(\lambda\xi)|\}
\]
and
\[
U_2=\{\xi\in\mathbb{R}:|x-x'+t\theta-t'\theta'|<2|t-t'||\Phi'(\lambda\xi)|\}.
\]
For $B_1$, we have
\begin{align*}
|\phi'(\lambda\xi)|&\ge\lambda|x-x'+t\theta-t'\theta'|-\lambda|t-t'||\Phi'(\lambda\xi)|\\
&\ge\frac\lambda2|x-x'+t\theta-t'\theta'|\\
&\ge \frac\lambda4|x-x'|\\
&>\lambda^{1-\sigma}\\
&\ge1
\end{align*}
because of \eqref{Shiraki condi:Thera to Omega}. From \eqref{Phi'' is bounded below} and the intermediate value theorem, $\Phi''(\xi)$ is single-signed on $(-\infty, -1]$ and $[1,\infty)$, which guarantees that $\Phi'(\xi)$ is monotone for $|\xi|\ge1$. Hence, $U_1$ consists of at most three intervals. Invoking Lemma \ref{lem:van der Corput}, 
\begin{align*}
B_1&\lesssim (\lambda|x-x'|)^{-1}\\
&\lesssim (\lambda|x-x'|)^{-\frac12}.
\end{align*}
On the other hand, for $B_2$, it follows from \eqref{Phi'' is bigger than Phi'} that
\begin{align*}
|\phi''(\lambda\xi)|&=\lambda^2|t-t'||\Phi''(\lambda\xi)|\\
&\gtrsim\lambda|t-t'||\Phi'(\lambda\xi)|\\
&\gtrsim\lambda|x-x'+t\theta-t\theta'|\\
&\gtrsim\lambda|x-x'|.
\end{align*}
Then, by using Lemma \ref{lem:van der Corput}, we obtain
\[
B_2\lesssim(\lambda|x-x'|)^{-\frac12}.
\]
Therefore, \eqref{ineq:K for V2} holds. 

\subsection*{The term $A_3$}
It remains to show 
\[
A_3\lesssim\lambda^{\frac12}\|F\|_{L^{q'}_xL^1_tL^1_\theta}^2.
\]
Trivially, 
\[
|K_\lambda(w,w')|\lesssim1
\]
so by the dual form of Young's convolution inequality
\begin{align*}
&\int_{-1}^1\int_{-1}^1\|F(x,\cdot,\cdot)\|_{L_t^1L_\theta^1}\|F(x',\cdot,\cdot)\|_{L^1_tL^1_\theta}\chi_{[-4\lambda^{-\sigma},4\lambda^{-\sigma}]}(x-x')\,dxdx'\\
&\lesssim\|F\|^2_{L^{q'}_xL^1_tL^1_\theta}\|\chi_{[-4\lambda^{-\sigma},4\lambda^{-\sigma}]}\|_{L^{\frac q2}}\\
&\sim\lambda^{1-\frac{2\sigma}{q}}\|F\|^2_{L^{q'}_xL^1_tL^1_\theta}.
\end{align*}
Therefore, we conclude that
\[
A_3\lesssim\lambda^{1-\frac{2\sigma}{q}}\|F\|^2_{L^{q'}_xL^1_tL^1_\theta}\lesssim\lambda^{\frac12}\|F\|^2_{L^{q'}_xL^1_tL^1_\theta}
\]
whenever \eqref{Shiraki condi:Thera to Omega}, as claimed.
\end{proof}
When $\lambda$ is large, by the definition of the upper Minkowski dimension, for small $\varepsilon>0$ there is  a constant $C_\varepsilon> 0 $ depending on $\varepsilon$ such that
\[
N(\Theta,\lambda^{-\sigma})\le C_\varepsilon\lambda^{\sigma\beta(\Theta)+\varepsilon}.
\]
Thus, if we also let $\widehat{\tilde{P_k}f}=\tilde{\psi_k}\widehat{f}$, where $\tilde{\psi}\in C_0^\infty((-4,-\frac14)\cup(\frac14,4))$ with $\tilde{\psi}\equiv1$ on $(-2,-\frac12)\cup(\frac12,2)$, then
\begin{align*}
\sum_{k\ge1}\sum_{j=1}^N\|M_{\Omega_{k,j}}P_kf\|_{L^q(I)}^q&=\sum_{k\ge1}\sum_{j=1}^N\|M_{\Omega_{k,j}}P_k\tilde{P_k}f\|_{L^q(I)}^q\\
&\lesssim\sum_{k\ge1} \sum_{j=1}^N2^{\frac{qk}{4}}\|\tilde{P_k}f\|_{L^2}^q\\
&\lesssim \sum_{k\ge1}2^{kq(\frac\sigma q\beta(\Theta)+\frac14+\frac\sigma q\varepsilon)}\|\tilde{P}_kf\|_{L^2}^q\\
&\sim \sum_{k\ge1}2^{-\frac\sigma q\varepsilon}\left(\int_{\supp{\tilde{\psi_k}}}2^{2k(\frac {\sigma\beta(\Theta)}{q}+\frac14+\frac{2\sigma\varepsilon}{q})}|\widehat{f}(\xi)|^2\,d\xi\right)^{\frac q2}\\
&\lesssim\|f\|_{H^{\frac{\sigma\beta(\Theta)}{q}+\frac14+\frac{2\sigma\varepsilon}{q}}}.
\end{align*}
Therefore, for arbitrary $\varepsilon>0$,
\[
\|M_\Theta f\|_{L^q(I)}\lesssim\|f\|_{H^{\frac{\sigma\beta(\Theta)}{q}+\frac14+\varepsilon}}
\]
holds. Referring to \eqref{Shiraki condi:Thera to Omega}, we shall let the Sobolev order be as small as possible to conclude that for arbitrary $\varepsilon>0$
\[
\|M_\Theta f\|_{L^q(I)}\lesssim\|f\|_{H^{\frac{\beta(\Theta)+1}{4}+\varepsilon}},
\]
which ends the proof.
\end{proof}
\remarks
The critical step in the above proof of Theorem \ref{thm:improved CLV} is Proposition \ref{prop: MPf is lesssim than lambda f}. The corresponding result in \cite{CLV12} (Lemma 3.1), stated for $q=2$ and $\Phi(\xi) = |\xi|^2$, is established through the following steps: $TT^*$ argument, the time localization lemma, Schur's lemma and then an oscillatory integral argument. Following this approach, one may extend by simple modification, for example, to the case $\Phi(\xi) = |\xi|^a$ with $a \geq \frac32$. The time localization lemma reduces to the case of time intervals of length $\lambda^{1-a}$, and for $a$ close to 1 this causes certain technical difficulties in the estimation of the oscillatory integrals which arise; in particular, the relationship \eqref{relationship: phase for A2} breaks down if we follow their argument as it stands. In order to overcome the significant technical difficulty, we removed the use of the time localization lemma and replaced this with appropriate decompositions of the domain $W \times W$.

\subsection*{Acknowledgment}The author would like to thank Neal Bez for his encouragement and a lot of discussions.

\end{document}